\documentclass{article}
\usepackage{mathrsfs}
\usepackage{graphicx}
\usepackage{amsmath}
\usepackage{amssymb}
\usepackage{txfonts}
\usepackage[all, pdf] {xy}
\usepackage{tikz}

\newtheorem{thm}{Theorem}[section]

\newtheorem{lem}[thm]{Lemma}

\begin{document}
	
	\begin{center}
		
		{\Large \bf On isomorphisms of tetravalent Cayley  graphs over generalized quaternion groups\footnote{Corresponding author: W. Meng ( mlwhappyhappy@163.com).
				W. Meng is supported by National Natural Science Foundation of China (12661001),  Guangxi Natural Science
				Foundation Program (2026GXNSFAA00641090)  and Guangxi Colleges and Universities Key Laboratory of Data Analysis and Computation.}}	
	\end{center}
	
	\vskip0.5cm
	
	\begin{center}
		 Lele Jin$^a$ , Li Ma$^b$, Wei Meng$^a$

		 a.  School of Mathematics and Computing Science,  Center for Applied Mathematics of Guangxi(GUET), Guilin University of Electronic Technology, Guilin, Guangxi, 541002, P.R. China.
		
		 b. School of Mathematics and Statistics, Qujing Normal University, Qujing, Yunnan, 655000, P.R. China.

		 E-mails:  1845973655@qq.com, 821962782@qq.com, mlwhappyhappy@163.com.
	\end{center}

	\begin{abstract}
		Let $G$ be a finite group and $S$ be a subset of $G$ with $1\not\in S$.  $S$ is called a $CI$-subset of $G$ if   for any $T\subset G$, whenever $Cay (G,S)\cong Cay (G,T)$, there is an element $\sigma\in Aut(G)$ such that $S^{\sigma}=T$.   For a positive integer $m$, $G$ is called an $m$-$CI$-group if all  subsets  $S$ of $G \setminus \{1\}$  with $|S|\leq  m$ and $S^{-1}=S$ are $CI$-subsets. In this paper, we show that  the generalized quaternion group of order $4n$ is a $4$-$CI$-group if and only if  $n$ is odd.
	\end{abstract}
	
	
	
	\section{Introduction}

Let $G$ be a finite group  and  $S$ be a subset of $G$ with $1\notin S$.  A Cayley digraph $Cay(G,S)$  is the direct graph  with vertex set $G$ and arc set $\{(g,sg)\mid g\in G,\ s\in S\}$.  In addition, $S=S^{-1}$, the pair of arcs $(g, sg)$ and $(sg, g)$ can be identified with the undirected edge $\{g, sg\}$,
in which case $Cay(G, S)$ is regarded as an undirected graph and referred to as a Cayley
graph. A Cayley (di)graph $Cay(G,S)$ is called a $CI$-(di)graph if  for any $T\subset G$, whenever $Cay(G, S) \cong Cay(G, T)$, there is an element $\sigma \in  Aut(G)$ such that $S^{\sigma} = T$. The group $G$ has the $m$-$DCI$-property
($m$-$CI$-property, resp.) if all Cayley digraphs (Cayley graphs, resp.) on $G$ of valency $m$ are
$CI$-digraphs ($CI$-graphs, resp.). By the valency of a Cayley digraph
$Cay (G, S)$ we mean its in-valency, which is equal to the cardinality $|S|$ and coincides with
its out-valency. Furthermore, $G$ is called an $m$-$(D)CI$-group if it has the $i$-$(D)CI$ property
for every $i \leq m$ and a $D(CI)$-group when it has the $i$-$(D)CI$ property for every positive
integer $i\leq |G|-1$.
	
	 In 1967, \'{A}d\'{a}m  \cite {ADA} conjectured that all finite cyclic groups
are $DCI$-groups.    Although this conjecture was disproved by Elspas and Turner\cite {ELS}, considerable energy has been devoted to seeking cyclic $DCI$-groups (see Babai\cite{BAB1}, Alspach
and Parsons \cite {ALS1}, and Godsil \cite {GOD}). Later, a complete classification of cyclic
$DCI$-groups was finally obtained by Muzychuk  \cite {MUZ1,MUZ2}, who proved that a cyclic group $C_n$ is a $DCI$-group if and only if $n =ab$, where $a |4$ and $b$ is a square-free odd number.  Moreover,  many results on abelian $(D)CI$-groups can be found in \cite {ALS2, FENG, HIR, KOV1,KOV2,KOV3, LI9, MUZ3, RYA, SOM}. Babai \cite{BAB1} proved that the dihedral groups  $D_{2p}$ of order $2p$ are non-abelian $CI$-groups, where $p$ is a prime.    Afterwards, Babai and Frankl \cite{BAB2, BAB3} deeply explored $CI$-groups of odd order and insoluble $CI$-groups. Li at al. extended the study of $CI$-groups and developed the theory of $(D)CI$-groups (see \cite {CON, LI3,LI4,LI5,LI6,LI7,LI8,LI10}.

On the other hand,  there are also some
results on $m$-$DCI$-groups with small $m$.  For example, Sun \cite {SUN} showed that every cyclic group is a $2$-$DCI$-group. Li \cite {LI1}
showed that every cyclic group is a $5$-$CI$-group.
 Fang and Xu \cite {FANG1,FANG2,FANG3} classified
abelian $m$-$DCI$-groups for every $m\in\{1,2,3\}$, and Qu and Yu \cite {QU} proved that for an
integer $n\geq3$ and $m\in\{1, 2, 3\}$, a dihedral group of order $2n$ is an $m$-$DCI$-group if
and only if $n$ is odd. Ma \cite{MA} showed that for $m \in \{1, 2, 3\}$, the generalized quaternion group $Q_{4n}(n\neq2)$ is an $m$-$DCI$-group if and only if $n$ is odd.

The study of finite groups with the $m$-$(D)CI$ property was initiated by Li \cite {LI2}, who  presented a necessary condition for cyclic groups with the
$m$-$DCI$ property and conjectured that the condition is also sufficient, which has been
confirmed by Dobson \cite {DOB}.  Li, Praeger and Xu \cite {LI5} gave a sufficient and necessary
condition for all finite abelian groups with the $m$-$DCI$ property for $1 \leq m \leq 4$,
where they proposed the following problem: characterize finite groups with the $m$-DCI
property.  Recently, combining the results of Li and Dobson,     Kov\'acs and   \v{S}inkovec \cite {KOV4} obtained a  complete  classification of cyclic $m$-$DCI$-groups and $m$-$CI$-groups.

In 2024, Xie, Feng and Kwon \cite{XIE1} gave a necessary condition for dihedral groups $D_{2n}$ to have the
$m$-$DCI$ property, and a sufficient and necessary condition when $n$ is a prime power. They showed that  $D_{2n}$ has the $m$-$DCI$ property for some $1 \le m \le n-1$, then $n$ is odd and is not divisible by the square of any prime smaller than $m$; and $D_{2p^e}$   has the $m$-$DCI$ property if
and only if either $e=1$, or $p$ is odd and $1 \leq  m \leq p$, where $p$ is an odd prime.    Furthermore, Xie, Lu and Feng \cite {XIE2}  obtained  a sufficient and necessary condition for dihedral groups $D_{2n}$ to have $4$-$DCI$-group:   $D_{2n}(n\geq3)$ is a $4$-$DCI$-group if and only if $n$ is odd and not divisible by $9$.

Recall that  a  generalized quaternion group $Q_{4n}$ of order $4n$  is defined by  the presentation:	
	\begin{center}
	 $Q_{4n} = \langle a, b \mid a^{2n} = 1,\ b^2 = a^n,\ b^{-1}ab = a^{-1} \rangle$, where $n\geq 2$.
	\end{center}
It is well-known that $Q_8$   is a  $DCI$-group.    Observe that Ma \cite {MA} have shown that for $m\in\{1,2,3\}$,  $Q_{4n}(n\neq2)$ is a $m$-$DCI$-group if and only if $n$ is odd.  Recently, Xie, Feng and Xia \cite {XIE3} investigated the $m$-$DCI$-property of $Q_{4n}$ and obtained  gave necessary conditions for the $m$-$DCI$ property of $Q_{4n}$. They showed that    if  $Q_{4n}(n\geq3)$ has  $m$-$DCI$-property  for $1\leq m\leq 2n-1$, then  $n$ is odd, and $n$ is not divisible by $p^2$ for any prime $p \leq m - 1$.  Moreover, they proved that if $Q_{4p^e}$ has $m$-$DCI$ property  for $1\leq m\leq 2p^e-1$ if and only if either $e=1$, or $m\leq p$, where $p$ is an odd prime.

 Note that Ma\cite {MA} also showed that  $Q_{4n}(n\neq)$ is a $3$-$CI$-group if and only if $n$ is an odd number.  Applying the results of Xie, Feng and Xia \cite {XIE3} and Ma \cite {MA}, we can see that if $Q_{4n}$ is a $4$-$DCI$-group, then $n$ is odd and $9\nmid n$, and $Q _{4p^e}$ is  a $4$-$DCI$-group for any $p>3$ is an odd prime. Recently, Xie and Zhang \cite {XIE4} obtained a new necessary and sufficient condition for $Q_{4n}$ to be a $4$-$CI$-group by investigating isomorphisms  of Bi-Cayley graph of $Q_{4n}$.

  Motivated by the above works,
  we investigate the isomorphic problem of tetravalent Cayley graphs on  $Q_{4n}$.   In this paper, we will give     a necessary and sufficient condition for $Q_{4n}$ to be a $4$-$CI$-group.  Our main results are  as follows.
	
	\begin{thm}
		The generalized quaternion group  $Q_{4n}$ is a $4$-$CI$-group if and only if either $n=2$, or $n$ is odd.
	\end{thm}

All unexplained notations and terminologies are standard  and can be found in \cite {BIG,HB}.  The paper is organized as follows. In Section 2, we introduce some preliminary results.  In Section 3, we prove Theorem 1.1.
	
	\section{Preliminaries}
	
	In this section, we collect some results which will be used in
the proof of the main results.

\begin{lem} \cite {LI5}
Let $G$ be a finite group, and let $S, T \subset G \setminus \{1\}$. Then
$Cay(G, S)\cong Cay(G, T)$ if and only if $Cay(\langle S\rangle, S) \cong Cay(\langle T\rangle, T)$.
\end{lem}
	
	  Let $X$  and $Y$  be two isomorphic Cayley
digraphs on $G$. We use $I(X,Y)$  to denote the set of all isomorphisms $\phi$  from
$X$ to $Y$  with $1^{\phi} = 1$.

	\begin{lem}\cite   {MA}
		Let $G$ be a finite group and $S, T$ subsets of $G \setminus \{1\}$ with $X = Cay(G, S) \cong Cay(G, T) = Y$.
		
		\noindent(1) If there exist $s \in S$, $t \in T$ such that $s^\phi = t$ holds for all $\phi \in I(X,Y)$, then
	
		\begin{center}
			$I(X,Y) \subseteq I(Cay(G, S \setminus \{s\}), Cay(G, T \setminus \{t\}))$.
		\end{center}
		
		\noindent(2) If there exists $x \in G$ such that $x^\phi = x$ holds for all $\phi \in I(X,Y)$, then
		\begin{center}
			$I(X,Y) \subseteq I(Cay(G, xS^2 \setminus \{1\}), Cay(G, xT^2 \setminus \{1\}))$,
		\end{center}
		where $xS^2 = \{xs_i s_j \mid s_i, s_j \in S\}$ and $xT^2 = \{xt_i t_j \mid t_i, t_j \in T\}$.		
	\end{lem}
	
	Motivated by Lemma 2.1(1), we obtain the following version.
	\begin{lem}
		Under the assumptions of Lemma 2.1, if there exist $S' \subset S$, $T' \subset T$ such that $(S')^\phi = T'$ holds for all $\phi \in I(X,Y)$, then
		\begin{center}
			$I(X,Y) \subseteq I(Cay(G, S \setminus S'), Cay(G, T \setminus T'))$.
		\end{center}
\end{lem}
\textbf{Proof.} Let $\phi\in I(X,Y)$. Observe that $R(g)\phi R((g^\phi)^{-1}) \in I(X,Y)$ for any $g \in G$. By assumption, for any $s'\in S'$, there exists an element $t'$ in $T'$, such that
		\begin{center}
			$s'^{R(g)\phi R((g^\phi)^{-1})} = t'$.
		\end{center}
	This yields to  $(s'g)^\phi = t'g^\phi$. Thus, for any $g,h \in G$,
\begin{center}
		$h  = s'g$ if and only if $h^\phi = (s'g)^\phi = t'g^\phi$.
\end{center}
		That is, $(g, h) \in E(Cay(G, S \setminus S')\bigr)$ if and only if
		$(g^\phi,h^\phi) \in E\bigl(Cay(G, T \setminus T')\bigr)$.
This shows that $\phi\in I(Cay(G, S \setminus S'), Cay(G, T \setminus T'))$. $\Box$

\begin{lem} \cite {LI1}
		All cyclic groups are $5$-CI groups.	
	\end{lem}

Recall that  a  generalized quaternion group $Q_{4n}$ of order $4n$  is defined by  the presentation:	
	\begin{center}
	 $Q_{4n} = \langle a, b \mid a^{2n} = 1,\ b^2 = a^n,\ b^{-1}ab = a^{-1} \rangle$, where $n\geq 2$.
	\end{center}		
The subgroup $A = \langle a\rangle$ is a normal subgroup of index $2$ in $G$, and the elements of
$G$ can be partitioned into two disjoint subsets of equal size as follows:
\begin{center}
$A = \{a^i | 0 \leq i \leq 2n - 1\}$ and $B = \{a^ib |0 \leq i \leq 2n - 1\}$.
\end{center}
It is seen that $o(a^ib)=4$ and $a^n$ is the unique involution of $Q_{4n}$.
The automorphism group of $Q_{4n}$ is given by
\begin{center}
$Aut(Q_{4n}) = \{\sigma_{i, j} : a^{\sigma_{i,j}}\rightarrow a^i, b^{\sigma_{i,j}}\rightarrow a^jb | 0 \leq i , j \leq 2n - 1, gcd(i , 2n) = 1\}$.
\end{center}

\begin{lem} \cite   {MA}
		For any $\pi \in Aut(\langle a \rangle)$ and for any integer $k$, define the mapping
		$\pi' : Q_{4n} \to Q_{4n}$ by
		\begin{center}
			$(a^i b^j)^{\pi'} = (a^i)^{\pi}(a^k b^j), \quad (i = 0, 1, \dots, 2n - 1, \, j = 0, 1)$.
		\end{center}
		Then $\pi' \in Aut(Q_{4n})$.
	\end{lem}

\begin{lem} \cite  {MA}
		The finite generalized quaternion group $Q_{4n}$ ($n \neq 2$) is an $m$-(D)CI-group for $m = 1, 2, 3$ if and only if $n$ is odd.
	\end{lem}

	\section{Proof of Theorem 1.1}
In this section,  we will prove Theorem 1.1.\\

\textbf{Proof of Theorem 1.1.} Let $G\cong Q_{4n}$ be the general quaternion group of order $4n$ with $n\geq2$. If $n=3$,
 then 	$G$ is   the
quaternion group of order $8$. It is well known that  $Q_8$ is a $CI$-group. So we can assume that $n\geq3$ is an odd number.
 Let $S$  be a subset of $G \setminus \{1\}$  such that $S^{-1}=S$ and $|S| = 4$.  We  need show that if $S$ is a $CI$-ubset of $G$, i.e. if
  $Cay (G,S)\cong Cay (G,T)$, then there exists an automorphism $\sigma$ of $G$, such that $S^{\sigma}=T$.

  Observe that  $a^n$ is the unique involution of $G$. Thus, $|S|=4$ implies that   $a^n\not\in S$. Furthermore, $S^{-1}=S$ implies that $|S\cap A|=4$, $2$ or $0$. Our proof will be divided into three cases as follows.\\
 		
 Case I: $ |S \cap A| = 4$.
	
In this case, we have $S\subseteq A$.  Since $A\leq G$,  we get that   $\langle S\rangle \leq A$. Note that $|A|=2n$ and $n$ is odd, we get $4 \nmid |\langle S \rangle|$. Furthermore, $Cay(G,S)\cong Cay(G,T)$  implies that
   $Cay (\langle S \rangle, S) \cong  Cay (\langle T \rangle, T)$ by Lemma 2.1. This leads to  $|\langle S \rangle| = |\langle T \rangle|$.  Since $4 \nmid |\langle T \rangle|$, $T\cap B=\emptyset$ and hence   $T \subset A$. Consequently, we have $\langle S\rangle = \langle T\rangle\leq A$. Applying Lemma 2.1 again, $Cay(A,S)\cong Cay(A,T)$. By Lemma 2.4, we know that $A$ is a $5$-$CI$-group.  So there exists $\alpha \in Aut (A)$ such that $S^\alpha = T$.  Applying Lemma 2.5,    $\alpha$ can be extended to an
automorphism $\sigma$ of $G$  by $a^{\sigma}=a^{\alpha}$ and $b^{\sigma}=b$. It is obvious that $S^\sigma=S^\alpha=T$. This shows that $S$ is a $CI$-subset of $G$.\\

Case II:  $|S \cap \langle a \rangle| = 2$.

 In this case, we can assume that $S=\{a^i, a^{-i}, a^jb, a^{n+j}b\}$, where $1\leq i\leq n-1$, $0\leq j\leq n-1$. By arguments of the Case I, we know that $T\cap B\neq\emptyset$.  Let $\Gamma=Cay(G,S)$. For any $u\in V(\Gamma)$, denote by $\Gamma(u)$ the neighborhood of $u$.  Since $n$ is an odd number, we have $a^{2i}\neq a^{-2i}$. Furthermore, it is easy seen that
	\begin{align*}
		&\Gamma(1) =S=\{a^i, a^{-i}, a^jb, a^{n+j}b\},\\
		&\Gamma(a^i) = \{1, a^{2i},  a^{j-i}b, a^{n+j-i}b\}, \quad
		\Gamma(a^{-i}) = \{1, a^{-2i}, a^{i+j}b, a^{n+i+j}b\},\\
		&\Gamma(a^jb) = \{1, a^n, a^{j-i}b, a^{i+j}b\}, \quad
		\Gamma(a^{n+j}b) = \{1,a^n, a^{n+j-i}, a^{n+i+j}b\}.
	\end{align*}
The induced subgraph $[\Gamma(1)\cup\Gamma(a^i)\cup\Gamma(a^{-i})\cup\Gamma(a^jb)\cup\Gamma(a^{n+j}b)]$ of $\Gamma(1)\cup\Gamma(a^i)\cup\Gamma(a^{-i})\cup\Gamma(a^jb)\cup\Gamma(a^{n+j}b)$ in  can be depicted as Fig. \ref{s1}.	
From Fig. \ref{s1}, it can be  seen that for any $\phi\in I(Cay(G,S),Cay(G,T))$,   $\{a^i, a^j\}^\phi\subset  T\cap A$.  Since $T\cap B\neq\emptyset$,  we have $|T\cap B|=2$ and $|T\cap A|=2$.  Without loss of generality, we can assume that
	 \begin{center}
	 	$S=\{a^i,a^{-i},b,a^nb\}$ and $T=\{a^u,a^{-u},a^vb,a^{n+v}b\}$.
	 \end{center}
By above arguments, we have
\begin{center}
 $\{a^i, a^{-i}\}^\phi=\{a^u,a^{-u}\}$ and $\{b,a^nb\}^{\phi}=\{a^vb,a^{n+v}b\}$.
 \end{center}
 Applying Lemma 2.3, we get
 \begin{center}
 $Cay (G,\{a^i,a^{-i}\})\cong Cay (G,\{a^u,a^{-u}\})$ and $Cay (G,\{b,a^nb\}\cong Cay (G,\{a^vb,a^{n+v}b\})$.
 \end{center}
 Moreover,  by Lemma 2.1, we have
  \begin{center}
  $Cay (\langle a^i\rangle ,\{a^i,a^{-i}\})\cong Cay (\langle a^u \rangle ,\{a^u,a^{-u}\})$.
    \end{center}
Thus $|\langle a^i\rangle|=|\langle a^u\rangle|$ and hence $\langle a^i\rangle = \langle a^u\rangle\leq A$. Applying Lemma 2.1 again,
we have
\begin{center}
$Cay (A,\{a^i,a^{-i}\})\cong Cay (A,\{a^u,a^{-u}\})$.
  \end{center}
  Since all cyclic groups are $2$-$CI$-group, there exists an automorphism $\alpha$ of $A$ such that $\{a^i, a^j\}^\alpha=\{a^u,a^{-u}\}$.   Applying Lemma 2.5,   $\alpha$ can be extended to an
automorphism $\sigma$ of $G$ by $(a^i)^\sigma=(a^i)^\alpha$ and $b^{\sigma}=a^vb$. It follows that $S^\sigma=T$ and hence $S$ is a $CI$-subset of $G$. \\
\begin{figure}[t]
\centering
\begin{tikzpicture}
\tikzstyle{main node}=[draw,shape=circle,label distance=0mm,inner sep=0pt];
\tikzstyle{every node}=[draw,shape=circle,label distance=-0.5mm,inner sep=1pt];
\node[main node] (1) at (0,0.8) [fill,label=above:\tiny{$1$}] {};
\node[main node] (ai) at (-2,0) [fill,label=left:\tiny{$a^i$}] {};
\node[main node]  (a-i) at (2,0) [fill,label=left:\tiny{$a^{-i}$}] {};
\node[main node] (ajb) at (-1,0) [fill,label=right:\tiny{$a^jb$}] {};
\node[main node] (anjb) at (1,0) [fill,label=right:\tiny{$a^{n+j}b$}] {};
\node[main node] (a2i) at (-3,-0.8) [fill,label=below:\tiny{$a^{2i}$}] {};
\node[main node] (a-2i) at (3,-0.8) [fill,label=below:\tiny{$a^{-2i}$}] {};
\node[main node]  (ai+jb) at (1,-0.8) [fill,label=below:\tiny{$a^{i+j}b$}] {};
\node[main node] (aj-ib) at (-2,-0.8) [fill,label=below:\tiny{$a^{j-i}b$}] {};
\node[main node] (ani+jb) at (2,-0.8) [fill,label=below:\tiny{$a^{n+i+j}b$}] {};
\node[main node] (anj-ib) at (-1,-0.8) [fill,label=below:\tiny{$a^{n+j-i}b$}] {};
\node[main node] (an) at (0,-0.8) [fill,label=below:\tiny{$a^{n}$}] {};
 \draw
 (1)--(ai)(1)--(a-i)(1)--(ajb)(1)--(anjb)
 (ai)--(a2i)(ai)--(aj-ib)(ai)--(anj-ib)
 (a-i)--(a-2i)(a-i)--(ai+jb)(a-i)--(ani+jb)
 (ajb)--(aj-ib)(ajb)--(ai+jb)(ajb)--(an)
 (anjb)--(ani+jb)(anjb)--(anj-ib)(anjb)--(an);
\end{tikzpicture}
\caption{The induced subgraph $[\Gamma(1)\cup\Gamma(a^i)\cup\Gamma(a^{-i})\cup\Gamma(a^jb)\cup\Gamma(a^{n+j}b)]$}\label{s1}
\end{figure}
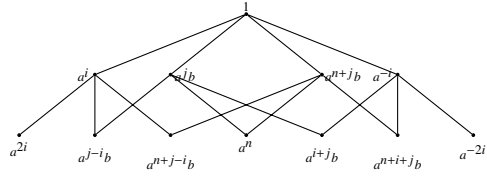		

Case III: $|S\cap A|=0$.

In this case, without loss of generality, let
\begin{center}
$S=\{b,a^nb,a^ib,a^{n+i}b\}$, where $1\leq i\leq n-1$.
\end{center} Let $T$ be a subset of $G$ such that
$Cay(G,S)\cong Cay(G,T)$. By the arguments of the Cases I-II, we can see that $T\cap A=\emptyset$. So we can assume that
\begin{center}
$T=\{a^ub,a^{n+u}b,a^vb,a^{n+v}b\}$, where $u\neq v$ and $0\leq u,v\leq n-1$.
\end{center}
Moveover, we have
\begin{center}
$S^2=\{1,a^n,a^i,a^{-i},a^{n+i},a^{n-i}\}$ and $T^2=\{1,a^n,a^{u-v},a^{n+u-v},a^{v-u},a^{n+v-u}\}$.
\end{center}
It is obvious that $1^\phi=1$ for any $\phi\in I(Cay(G,S),Cay(G,T))$. Applying Lemma 2.2(2), we have
$Cay(G,S^2\setminus\{1\})\cong Cay (G,T^2\setminus\{1\})$, i.e.,
\begin{center}
		 $Cay(G,\{a^n,a^i,a^{-i},a^{n+i},a^{n-i}\})\cong Cay (G,\{a^n,a^{u-v},a^{n+u-v},a^{v-u},a^{n+v-u}\})$.
\end{center}
Furthermore, by Lemma 2.1, we have
\begin{center}
$Cay (A,\{a^n,a^i,a^{-i},a^{n+i},a^{n-i}\})\cong Cay (A,\{a^n,a^{u-v},a^{n+u-v},a^{v-u},a^{n+v-u}\})$.
\end{center}
We have known  that $A$ is a $5$-$CI$-group by Lemma 2.4. So there exists an automorphism $\alpha$ of $A$ such that
	\begin{center}
$\{a^n,a^i,a^{-i},a^{n+i},a^{n-i}\}^\alpha=\{a^n,a^{u-v},a^{n+u-v},a^{v-u},a^{n+v-u}\}$.
\end{center}
Note that $a^n$ is the unique involution in $A$, we have $(a^n)^\alpha=a^n$ and consequently
\begin{center}
$\{a^i,a^{-i},a^{n+i},a^{n-i}\}^\alpha=\{a^{u-v},a^{n+u-v},a^{v-u},a^{n+v-u}\}$.
\end{center}
Without loss of generality, we can let $(a^i)^\alpha=a^{u-v}$. Then
\begin{center}
$(a^{-i})^\alpha=a^{v-u}$, $(a^{n+i})^\alpha=a^{n+u-v}$, $(a^{n-i})^\alpha=a^{n+v-u}$.
\end{center}
Furthermore, applying Lemma 2.5, $\alpha$ can be extended to an
automorphism $\sigma$ of $G$ by $(a^i)^\sigma=(a^i)^\alpha$ and $b^{\sigma}=a^vb$. It follows that $S^\sigma=T$ and hence $S$ is a $CI$-subset of $G$.

Conversely, suppose that $G$ is a $4$-$CI$-group. In particular, $G$ is also a $3$-$CI$-group, so either $n=2$, or $n\geq3$ is odd by Lemma 2.6.  The proof of theorem 1.1 is complete. $\Box$\\

	\noindent{\bf Declaration of Interest}
	
	The authors declared there don't exist potential competing interests.

	\bibliographystyle{amsplain}

\begin{thebibliography}{}
		
		
		\bibitem{ADA}	A. \'{A}d\'{a}m, 'Research problem 2-10', J. Combin. Theory, 1967, 2:309.
		
		\bibitem{ALS1} B. Alspach, T.D. Parsons, Isomorphism of circulant graphs and digraphs, Discrete Math., 1979, 25:97-108.
		
		\bibitem{ALS2} B. Alspach, Isomorphisms of Cayley graphs on abelian groups, NATO ASI Ser. C, 1997, 497:1-22.
		
		\bibitem{BAB1} L. Babai, Isomorphism problem for a class of point-symmetric structures, Acta Math. Acad. Sci. Hungar., 1977, 29:329-336.
		
		\bibitem{BAB2} L. Babai, P. Frankl, Isomorphisms of Cayley graphs I, Colloq. Math. Soc. J\'anos Bolyai, 1978, 18:35-52.
		
		\bibitem{BAB3} L. Babai, P. Frankl, Isomorphisms of Cayley graphs II, Acta Math. Acad. Sci. Hungar., 1979, 34:177-183.

	\bibitem{BIG} N. Biggs , Algebraic Graph Theory, Cambridge University Press, New York, 1993. 
		
		\bibitem{CON} M. Conder, C.H. Li, On isomorphisms for finite Cayley graphs, European J. Combin., 1998, 19:911-919.
		
		
		\bibitem{DOB} E. Dobson, On the Cayley isomorphism problem, Discrete Math., 2002, 247:107-116.
		
		\bibitem{ELS} B. Elspas, J. Turner, Graphs with circulant adjacency matrices, J. Combin. Theory, 1970, 9:297-307.
		
		\bibitem {FANG1} X.G. Fang, A characterization of finite abelian 2-DCI groups, J. Math., 1988, 8:315-317.
		
		
		\bibitem {FANG2} X.G. Fang, M.Y. Xu, Abelian 3-DCI groups of odd order, Ars Combin., 1988, 28:247-251.
		
		\bibitem {FANG3} X.G. Fang, Abelian 3-DCI groups, Ars Combin., 1992, 32:263-267.
		
		\bibitem{FENG} Y.Q. Feng, I. Kov\'acs, Elementary abelian groups of rank 5 are DCI-groups, J. Comb. Theory Ser. A, 2018, 157:162-204.
		
		\bibitem{GOD} C.D. Godsil,  On Cayley graph isomorphisms, Ars Combin., 1983, 15:231-246.
		
		\bibitem {HIR} M. Hirasaka, M. Muzychuk, An elementary abelian group of rank 4 is a CI-group, J. Comb. Theory Ser. A, 2001, 94:339-362.
		
		\bibitem{HB}  B. Huppert, \textit{Endliche Gruppen I}, Springer-Verlag, Berlin, Heidelberg, New York, 1967.
		
		\bibitem {KOV1} I. Kov\'acs, M. Muzychuk, The group $Z_{p}^2\times Z_q$ is a CI-group, Commun. Algebra, 2009, 37:3500-3515.
		
		\bibitem{KOV2} I. Kov\'acs, G. Ryabov, CI-property for decomposable Schur rings over an abelian group, Algebra Colloq., 2019, 26:147-160.
		
		\bibitem {KOV3} I. Kov\'acs, G. Ryabov, The group $C_{p}^4\times C_q$ is a DCI-group, Discrete Math., 2022, 345:112705.
		
		\bibitem{KOV4} I. Kov\'acs, L. \v{S}inkovec, Cyclic $m$-DCI-groups and $m$-CI-groups, Art Discrete Appl. Math., 2025, 8(1):B72.
		
		\bibitem{LI1} C.H. Li, Isomorphisms and classification of Cayley graphs of small valencies on finite abelian groups, Australas. J. Combin., 1995, 12:3-14.
		
		\bibitem{LI2} C.H. Li, The Cyclic groups with the m-DCI Property, European J. Combin., 1997, 18:655-665.
		
		\bibitem{LI3} C.H. Li,  C.E.Praeger,  On finite groups in which any two elements of the same order are fused or
		inverse-fused, Comm. Algebra, 1997, 25:3081-3118.
		
		\bibitem{LI4} C.H. Li, C.E.Praeger, M.Y. Xu, Isomorphisms of finite Cayley digraphs of bounded valency, J.
		Combin. Theory Ser. B, 1998, 73:164-183.
		
		\bibitem{LI5} C.H. Li, C.E. Praeger,  M.Y. Xu, On finite groups with the Cayley isomorphism property, J. Graph
		Theory, 1998, 28:21-31.
		
		\bibitem{LI6} C.H. Li, Isomorphisms of finite Cayley digraphs of bounded valency II, J. Combin. Theory Ser. A, 1999, 88:19-35.
		
		\bibitem{LI7} C.H.  Li,  Finite CI-groups are soluble, Bull. Lond. Math. Soc., 1999, 31:419-423.
		
		\bibitem{LI8}  C.H.  Li,  C.E. Praeger,  On the isomorphism problem for finite Cayley graphs of bounded valency,
		European J. Combin., 1999, 20:279-292.
		
		\bibitem{LI9} C.H. Li, On isomorphisms of finite Cayley graphs--a survey, Discrete Math., 2002, 256:301-334.
		
		
		\bibitem{LI10}C.H. Li, Z.P. Lu, P.P. P\'alfy, Further restrictions on the structure of finite CI-groups, J. Algebraic Combin., 2007, 26:161-181.
		
		\bibitem{MA} H.C. Ma, On isomorphisms of Cayley digraphs on dicyclic groups, Australas. J. Combin., 1997, 16:189-194.
		
		\bibitem{MUZ1} M. Muzychuk, \'{A}d\'am's conjecture is true in the square-free case, J. Combin. Theory Ser. A, 1995, 72:118-134.
		
		\bibitem{MUZ2} M. Muzychuk, On \'{A}d\'am's conjecture for circulant graphs, Discrete Math., 1997, 176:285-298.
		
		\bibitem{MUZ3} M. Muzychuk, An elementary abelian group of large rank is not a CI-group, Discrete Math., 2003, 264:167-185.
		
		\bibitem{P}P.P. P\'alfy, Isomorphism problem for relational structures with a cyclic automorphism, European J. Combin., 1987, 8:35-43.
		
		\bibitem{QU}H. Qu, J. Yu, On isomorphisms of Cayley digraphs on dihedral groups, Australas. J. Combin., 1997, 15:213-220.		
		
		\bibitem {RYA} G. Ryabov, The Cayley isomorphism property for the group $C_4\times C_{2p}$, Commun. Algebra, 2021, 49:1788-1804.
		
		\bibitem {SOM} G. Somlai, Elementary Abelian p-groups of rank $2p +3$ are not CI-groups, J. Algebraic Comb., 2011, 34:323-335.
		
		\bibitem{SUN} L. Sun, Isomorphisms of circulants with degree 2, J. Beijing Ins. Technol., 1984, 9:42-46.
		
		\bibitem{TOI} S. Toida, A note on \'{A}d\'{a}m's conjecture, J. Combin. Theory Ser. B, 1977, 23:239-246.
		
		\bibitem{XIE1} J.H. Xie, Y.Q. Feng, Y.S. Kwon, Dihedral groups with the m-DCI property, J. Algebraic Combin., 2024, 60:73-86.
		
		\bibitem{XIE2} J.H. Xie, Z.P. Lu, Y.Q. Feng, On isomorphisms of tetravalent cayley digraphs over dihedral groups, J. Algebraic Comb., 2025, 62:24-39.
		
		\bibitem{XIE3} J.H. Xie, Y.Q. Feng, B. Xia, Generalized quaternion groups with the $m$-DCI property, Electron. J. Combin., 2025, 32(2):P2.47.

	\bibitem{XIE4} J.H. Xie, Z. Zhang, Isomorphisms of bi-Cayley graphs on generalized quaternion groups. Commun.  Algebra, 2026, 54(2):459-470
		
		
		
		
		
		
		
		
	\end{thebibliography}

\end{document}